\documentclass{amsart} 
\usepackage{bbm}
\usepackage[top=2.8cm,left=2.8cm,right=2.8cm,bottom=2.6cm]{geometry}
\usepackage{amsfonts}
\usepackage{amsmath,amsthm,amscd,mathrsfs,amssymb,graphicx,enumerate,
  dsfont}
\usepackage{xypic}

\newtheorem{thm}{Theorem}[section]
\newtheorem{cor}[thm]{Corollary}
\newtheorem{lem}[thm]{Lemma}
\newtheorem{prop}[thm]{Proposition}
\newtheorem{conj}[thm]{Conjecture}
\theoremstyle{definition}
\newtheorem{defn}[thm]{Definition}

\newtheorem{rmk}[thm]{Remark}

 \DeclareMathOperator{\Spec}{Spec}

\DeclareMathOperator{\End}{End}

\newcommand{\C}{\ensuremath\mathds{C}}

\newcommand{\Z}{\ensuremath\mathds{Z}}
\newcommand{\Q}{\ensuremath\mathds{Q}}

\newcommand{\PP}{\ensuremath\mathds{P}}

\newcommand{\HH}{\ensuremath\mathrm{H}}
\newcommand{\CH}{\ensuremath\mathrm{CH}}

\newcommand{\Set}[2]{\left\{#1 :#2\right\}}

\begin{document}

\title{Birational Chow--K\"unneth decompositions}
\author{Mingmin Shen}



\address{
KdV Institute for Mathematics, University of Amsterdam, P.O.Box 94248, 1090 GE Amsterdam, Netherlands}
\email{M.Shen@uva.nl}

\date{\today}

\begin{abstract} 
We study the notion of a birational Chow--K\"unneth decomposition, which is essentially a decomposition of the integral birational motive of a variety. The existence of a birational Chow--K\"unneth decomposition is stably birationally invariant and this notion refines the Chow theoretical decomposition of the diagonal. We show that a birational Chow--K\"unneth decompostion exists for the following varieties: (a) Jacobian variety; (b) Hilbert scheme of points on a $K3$ surface and (c) The variety of lines on a stably rational cubic threefold or a stably rational cubic fourfold.
\end{abstract}

\maketitle

\section{Introduction}
The concept of a Chow--K\"unneth decomposition was introduced by Murre \cite{murre}. It is very important to the understanding of the motive of a smooth projective variety. Corresponding to a birational motive as was introduced by Kanh and Sujatha \cite{kahn}, there is the notion of a birational Chow--K\"unneth decomposition. In this article, we study the birational Chow--K\"unneth decomposition with integral coefficients from a very down-to-earth point of view. We insist on integral coefficients because we hope that this notion will give a criteria to distinguish special varieties (\textit{e.g.} Jacobian varieties, Hilbert schemes of a $K3$ surface, ...) from their general deformations. We show that this notion has an interesting connection with the rationality problem of cubic fourfolds.

Let $X$ be a smooth projective variety of dimension $d$ over the field $\C$ of complex numbers. The classical K\"unneth formula gives 
\[
[\Delta_X] = \pi^0_{X,\mathrm{hom}} + \pi^1_{X,\mathrm{hom}}+\cdots + \pi^{2d}_{X,\mathrm{hom}}
\]
where $\pi^p_{X,\mathrm{hom}} \in \HH^{2d-p}(X,\Z)\otimes\HH^p(X,\Z)$. An \textit{integral Chow--K\"unneth decomposition} is
\[
\Delta_X = \pi_X^0 +\pi_X^1 +\cdots +\pi_X^{2d},\quad \text{in }\CH^d(X\times X),
\]
such that
\begin{enumerate}
\item The element $\pi^i_X\in \CH^d(X\times X)$, $0\leq i\leq 2d$, lifts the corresponding cohomological projector, namely $[\pi_X^i] = \pi^i_{X,\mathrm{hom}}$ in $\HH^{2d}(X\times X,\Z)$.
\item The collection $\{\pi^i_X\}$ satisfies $\pi^i_X\circ\pi^j_X=0$ for all $0\leq i\neq j\leq 2d$ and $\pi^i_X\circ\pi^i_X = \pi^i_X$ for all $0\leq i\leq 2d$.
\end{enumerate}
If we allow $\Q$-coefficients, then we get the usual \textit{Chow--K\"unneth decomposition}.
\begin{conj}[Murre \cite{murre}] \label{conj murre}
A Chow--K\"unneth decomposition (with $\Q$-coefficients) exists for each $X$.
\end{conj}
There are two other formulations of an integral Chow--K\"unneth decomposition. The first one says that the integral Chow motive $\mathfrak{h}(X)$ decomposes as a direct sum of $\mathfrak{h}^i(X) = (X,\pi_X^i)$, $0\leq i\leq 2d$. The second one is a decomposition of $\CH_d(X\times X)$ into a direct sum of sub-algebras. 

To define the notion of a birational Chow--K\"unneth decomposition, we first replace $\CH_d(X\times X)$ by $\CH_0(X_K)$, where $K=\C(X)$ is the function field of $X$. It turns out that the composition operation on $\CH_d(X\times X)$ restricts to a composition on $\CH_0(X_K)$. The diagonal $\Delta_X$ restricts to a unit $\delta_X\in\CH_0(X_K)$. Kahn and Sujatha \cite{kahn} associate a birational motive $\mathfrak{h}^{o}(X)$ to each smooth projective variety $X$. Then the algebra $\langle \CH_0(X_K), \delta_X, \circ\rangle$ is the endormorphism algebra of the birational motive of $X$. The usual cycle class map restricts to a \textit{birational cycle class map} 
\[
[-]: \CH_0(X_K)\longrightarrow \bigoplus_{i=0}^d \End(\HH^{i,0}(X)).
\]
Let $\varpi^i_{X,\mathrm{hom}}$ be the projector onto the factor $\HH^{i,0}(X)$. Then a birational Chow--K\"unneth decomposition is an equation
\[
\delta_X = \varpi_X^0 +\varpi_X^1 + \cdots + \varpi_X^d,\quad \text{in }\CH_0(X_K)
\]
with $\varpi_X^i\in \CH_0(X_K)$, $0\leq i\leq d$, being projectors lifting the corresponding cohomological ones; see Definition \ref{defn birational CK}. In terms of birtional motives, a birational Chow--K\"unneth decomposition is equivalent to a decomposition 
\[
 \mathfrak{h}^o(X) = \mathfrak{h}^{o,0}(X)\oplus \mathfrak{h}^{o,1}(X) \oplus \cdots \oplus \mathfrak{h}^{o,d}(X)
\]
where $\mathfrak{h}^{o,i}(X)=(\mathfrak{h}(X),\varpi_X^i)$. We emphasize that a birational Chow--K\"unneth decomposition has integral coefficients. Occasionally, we may also consider birational Chow--K\"unneth decompositions with other coefficients (\textit{e.g.} $\Q$, $\Z[\frac{1}{2}]$). Conjecture \ref{conj murre} predicts that a birational Chow--K\"unneth decomposition with $\Q$-coefficients exists for each $X$. 

The main results can be summarized as follows.
\begin{enumerate}
\item The existence of a birational Chow--K\"unneth decomposition is stably birationally invariant. (Proposition \ref{prop birational invariance})
\item A Jacobian variety admits a birational Chow--K\"unneth decomposition. (Theorem \ref{thm jac})
\item The Hilbert scheme $S^{[n]}$ of a $K3$-surface $S$ admits a birational Chow--K\"unneth decomposition. (Theorem \ref{thm Hilb})
\item The variety of lines on a smooth cubic threefold or a smooth cubic fourfold admits a birational Chow--K\"unneth decomposition with $\Z[\frac{1}{2}]$-coefficients. (Theorem \ref{thm generic cubic})
\item The variety of lines on a stably rational cubic threefold or a stably rational cubic fourfould admits a birational Chow--K\"unneth decomposition. (Theorem \ref{thm rational cubic})
\end{enumerate}

These results naturally raise the following questions. (a) Does a generic principally polarized abelian variety of dimension at least 4 have a birational Chow--K\"unneth decomposition? (b) Does a generic hyperK\"ahler variety deformation equivalent to $\mathrm{Hilb}^n(K3)$, $n\geq 2$, have a birational Chow--K\"unneth decomposition?

\textit{Achknowledgement}. I would like to thank Charles Vial for helpful discussions.

\section{Composition and product of zero cycles over the function field}

Let $X$ be a smooth projective variety of dimension $d$ over a field $k$. This section is devoted to the study of the algebra $\CH_0(X_K)$ of 0-cycles on $X$ over its function field $K=k(X)$. This algebra is the endomorphism algebra of the birational motive $\mathfrak{h}^o(X)$ of $X$. Most results of this section can essentially be found in \cite{kahn}. However, our approach does not use the notion of a birational motive. We prove, among others, the stably birational invariance of this algebra. 

\begin{lem}\label{lem support lemma}
Let $Y$ and $Z$ be smooth projective varieties. Let $\Gamma_1 \in \CH^r(X \times Y)$ and $\Gamma_2 \in \CH^s(Y\times Z)$ be two correspondences. Then the following statements hold.
\begin{enumerate}[(i)]
\item Assume that $\Gamma_1$ factors through a subset $Y'\subset Y$ of codimension $n$. Then $\Gamma_2\circ \Gamma_1\in \CH^{r+s-d_Y} (X\times Z)$ factors through a subset $Z'\subset Z$ of codimension at least $n+s-d_Y$.
\item Assume that $\Gamma_2$ is supported on $Y'\times Z$ for some closed subset $Y'\subset Y$ of codimension $n$. Then $\Gamma_2\circ\Gamma_1$ is supported on $X'\times Z$ where $X'\subset X$ is a closed subset of codimension $n+r-d_Y$.
\item Assume that $\Gamma_1$ is supported on $X'\times Y$ for some closed subset $X'\subset X$, then $\Gamma_2\circ \Gamma_1$ is supported on $X'\times Z$.
\end{enumerate}
\end{lem}
\begin{proof}
Note that statement (ii) follows from statement (i) by taking the transpose and that statement (iii) is trivial. Hence we only need to show (i). By definition, the composition $\Gamma_2\circ \Gamma_1$ factors through
\[
 p_Z\big( |\Gamma_2|\cap (Y'\times Z) \big)
\]
where $p_Z:Y\times Z\rightarrow Z$ is the projection and $|\Gamma_2|$ denotes the support of the cycle $\Gamma_2$. By moving the cycle $\Gamma_2$, we may assume that $\Gamma_2$ intersects $Y'\times Z$ properly. Then the above set has codimension at least $n+s-d_Y$.
\end{proof}

The above lemma allows us to compose 0-cycles over function fields.

\begin{lem}\label{lem composing 0-cycles}
Let $X$, $Y$ and $Z$ be smooth projective varieties with function fields $K$, $L$ and $M$ respectively. Then there is a composition
\[
 \CH_0(X_L) \times \CH_0(Y_M) \longrightarrow \CH_0(X_M),\qquad (\gamma_1,\gamma_2)\mapsto (\Gamma_2\circ\Gamma_1)|_{X\times \Spec M},
\]
where $\Gamma_1 \in \CH^{d_X}(X\times Y)$ and $\Gamma_2\in \CH^{d_Y}(Y\times Z)$ are spreadings of $\gamma_1$ and $\gamma_2$ respectively.
\end{lem}

\begin{proof}
We only need to show that the composition defined above is independent of the choice of the spreadings. Let $\Gamma'_1\in \CH^{d_X}(X\times Y)$ and $\Gamma'_2\in \CH^{d_Y}(Y\times Z)$ be two other spreadings. Then we can write
\[
 \Gamma'_1 = \Gamma_1 + \Delta_1 \qquad\text{and} \qquad \Gamma'_2 = \Gamma_2 + \Delta_2
\]
where $\Delta_1\in \CH^{d_X}(X\times Y)$ is supported on $X\times D_1$ for some divisor $D_1\subset Y$ and $\Delta_2\in \CH^{d_Y}(Y\times Z)$ is supported on $Y\times D_2$ for some divisor $D_2\subset Z$. Then we have
\[
\Gamma'_2\circ \Gamma'_1= \Gamma_2\circ \Gamma_1 + \Delta_2\circ \Gamma_1 + \Gamma'_2\circ \Delta_1.
\]
By Lemma \ref{lem support lemma}, we know that both $\Delta_2\circ\Gamma_1$ and $\Gamma'_2\circ \Delta_1$ factors through a divisor of $Z$. Hence
\[
 (\Gamma'_2\circ\Gamma'_1 )|_{X\times \Spec M} = (\Gamma_2\circ\Gamma_1)|_{X\times \Spec M}, \quad\text{in }\CH_0(X_M).
\]
This proves the lemma.
\end{proof}

\begin{defn}
Let $\gamma_1,\gamma_2\in \CH_0(X_K)$ be two 0-cycles on $X_K$. By Lemma \ref{lem composing 0-cycles} we can define $\gamma = \gamma_1\circ \gamma_2\in\CH_0(X_K)$. The diagonal $\Delta_X$ restricts to an element $\delta_X\in \CH_0(X_K)$ which satisfies $\gamma\circ\delta_X = \delta_X\circ \gamma = \gamma$ for all $\gamma\in \CH_0(X_K)$. This gives a possibly non-commutative $\Z$-algebra structure on $\CH_0(X_K)$. This algebra will be denoted $\langle \CH_0(X_K), \delta_X, \circ \rangle$ or simply $\CH_0(X_K)$ if there is no confusion.
\end{defn} 

Let $X'$ be a smooth projective variety of dimension $d'$ and let $K'=k(X')$ be the corresponding function field. The following proposition is useful to relate $\CH_0(X_K)$ to $\CH_0(X'_{K'})$.

\begin{prop}\label{prop transfer 0 cycles}
Let $\gamma_1\in \CH_0(X_{K'})$ and $\gamma_2\in \CH_0(X'_K)$. There is a well-defined group homomorphism
\[
 \gamma_2\circ(-)\circ \gamma_1: \CH_0(X'_{K'}) \longrightarrow \CH_0(X_K)
\]
given by
\[
\gamma' \mapsto \Gamma_2\circ \Gamma'\circ \Gamma_1|_{X_K},\qquad \forall \gamma'\in\CH_0(X'_{K'}),
\]
where the correspondence $\Gamma'\in \CH_{d'}(X'\times X')$ is a spreading of $\gamma'$, the cycle $\Gamma_1\in \CH_{d'}(X\times X')$ is a spreading of $\gamma_1$ and the cycle $\Gamma_2\in \CH_d(X'\times X)$ is a spreading of $\gamma_2$.
\end{prop}

\begin{proof}
This is a direct consequence of Lemma \ref{lem composing 0-cycles}.
\end{proof}

\begin{cor}\label{cor generic finite}
Let $f:X'\dashrightarrow X$ be a generically finite rational map between smooth projective varieties. 
\begin{enumerate}[(i)]
\item there is a homomorphism $f^\#: \CH_0(X_K)\longrightarrow \CH_0(X'_{K'})$ defined by
\[
 f^\# \gamma = f^*\circ \gamma \circ f_* 
\]
where $f_*$ is the closure $\bar{\Gamma}_f$ of the graph $\Gamma_f$ restricted to $X'\times \Spec K$ and $f^*$ is the restriction of the transpose ${}^t\bar{\Gamma}_f$ to $X\times \Spec K'$. Moreover,
\[
 f^\#\gamma_1 \circ f^\#\gamma_2 = \deg(f)\, f^\#(\gamma_1\circ\gamma_2)
\]
for all $\gamma_1,\gamma_2\in\CH_0(X_K)$.

\item There is a homomorphism $f_{\#}: \CH_0(X'_{K'})\longrightarrow \CH_0(X_K)$ defined by 
\[
 f_{\#}\gamma' = f_* \circ \gamma' \circ f^*.
\]
If $f$ is of degree one (i.e. birational), then
\[
 f_{\#}\gamma'_1 \circ f_{\#}\gamma'_2 = f_{\#}(\gamma'_1\circ\gamma'_2)
\]
for all $\gamma'_1,\gamma'_2\in \CH_0(X'_{K'})$. 

\item If $f$ is birational, then $f^\#$ and $f_\#$ are inverse to each other and both are isomorphisms of $\CH_0(X_K)$ and $\CH_0(X'_{K'})$ as algebras.
\end{enumerate}
\end{cor}
\begin{proof}
For \textit{(i)}, we only need to verify the formula for composition, which follows easily from the fact that $f_*\circ f^* =\deg(f)$. For \textit{(ii)}, we oberserve that, when $f$ is birational, we have
\[
 {}^t\bar{\Gamma}_f\circ \bar{\Gamma}_f = \Delta_{\tilde{X}} + \Gamma',
\]
for some $\Gamma'$ supported on $D\times D$ with $D\subset X'$ being a divisor. By restricting to $X'_{K'}$, we obtain $f^*\circ f_* = \delta_{X'}$ in $\CH_0(X'_{K'})$. The last statement follows from the extra relation $f^\#\delta_X = \delta_{X'}$ and $f_\# \delta_{X'} = \delta_X$.
\end{proof}

Let $X$ and $Y$ be smooth projective varieties with function fields $K$ and $L$ respectively. Let $\gamma\in \CH_0(X_K)$ and $\gamma'\in \CH_0(Y_L)$. We can spread these cycles and get $\Gamma\in \CH_{d_X}(X\times X)$ and $\Gamma'\in\CH_{d_Y}(Y\times Y)$. Then we can define
\[
 \Gamma\otimes\Gamma' := p_{13}^*\Gamma \cdot p_{24}^*\Gamma' \in \CH_{d_X + d_Y}(X\times Y\times X\times Y).
\]
We define
\[
 \gamma\otimes \gamma' :=(\Gamma\otimes\Gamma')|_{X\times Y \times\eta_{X\times Y}}\in \CH_0\big( (X\times Y)_{K\otimes L} \big)
\]
as the \textit{product} of $\gamma$ and $\gamma'$. It is clear that this definition is independent of the choice of the spreadings. Thus we have a bilinear map
\begin{equation}\label{eq tensor product}
\otimes : \CH_0(X_K)\times \CH_0(Y_L) \longrightarrow \CH_0 \big( (X\times Y)_{K\otimes L} \big).
\end{equation}

\begin{prop}\label{prop product commutes with composition}
Let $X$ and $Y$ be smooth projective varieties with function fields $K$ and $L$ respectively. Then
\[
(\gamma_1\otimes\gamma'_1) \circ (\gamma_2\otimes\gamma'_2) = (\gamma_1\circ\gamma_2) \otimes (\gamma'_1\circ\gamma'_2),\quad \text{in }\CH_0 \big( (X\times Y)_{K\otimes L} \big),
\]
for all $\gamma_1,\gamma_2\in \CH_0(X_K)$ and $\gamma'_1,\gamma'_2\in\CH_0(Y_L)$.
\end{prop}

\begin{proof}
We easily check that the same equation holds when all cycles are replaced by their spreadings. Then we restrict the equation to the generic point of $X\times Y$.
\end{proof}

The special case of $Y=\PP^1_k$ is of particular interest. Note that $\CH_0(\PP^1_{k(\PP^1)})$ is free of rank one with generator $x\times \eta_{\PP^1}$ where $x\in\PP^1$ is a closed point. In this case we have the following proposition.

\begin{prop}\label{prop stable birational CH0}
Let $X$ be a smooth projective variety of dimension $d$ with function field $K$.
\begin{enumerate}[(i)]
\item The homomorhism
\[
\varphi_X : \CH_0(X_K)\longrightarrow \CH_0\big( (X\times \PP^1)_{K\otimes k(\PP^1)} \big), \qquad
\gamma \mapsto \gamma\otimes (x\times {\eta_{\PP^1}}),
\]
is an isomorphism which is compatible with the algebra structures. Namely,
\begin{align*}
\varphi_X(\delta_X) & = \delta_{X\times\PP^1}\quad\text{and}\\
 \varphi_X(\gamma_1\circ\gamma_2) & = \varphi_X(\gamma_1)\circ\varphi_X(\gamma_2), \quad \text{in } \CH_0\big( (X\times\PP^1)_{K\otimes k(\PP^1)} \big),
 \end{align*}
for all $\gamma_1,\gamma_2\in \CH_0(X_K)$.

\item The algebra $\langle \CH_0(X_K), \delta_X, \circ \rangle$ is a stably birational invariant.
\end{enumerate}
\end{prop}

\begin{proof}
We first note that a 0-cycle on $(X\times \PP^1)_{K\otimes k(\PP^1)}$ is obtained as the restriction of a $(d+1)$-cycle on $X\times\PP^1\times X\times \PP^1$. One computes
\begin{align*}
 \CH_{d+1}(X\times\PP^1 \times X\times \PP^1) = & \CH_{d+1}(X\times X) \otimes (x,x) \\
 &\oplus \CH_d(X\times X) \otimes (x\times \PP^1) \\
 & \oplus \CH_d(X\times X)\otimes (\PP^1\times x)\\
 &\oplus \CH_{d-1}(X\times X) \otimes (\PP^1\times \PP^1)
\end{align*}
In the above equation, only the summand $\CH_d(X\times X)\otimes (x\times \PP^1)$ may have nonzero restriction to $(X\times \PP^1)_{K\otimes k(\PP^1)}$ and the restriction factors through $\CH_0(X_K)\otimes (x\times \eta_{\PP^1})$. This shows that $\varphi_X$ is surjective. Let $\gamma\in \CH_0(X_K)$ be a cycle such that $\varphi_X(\gamma)=0$. Take a spreading $\Gamma\in \CH_d(X\times X)$ of $\gamma$.
Then $p_{13}^*\Gamma \cdot p_2^*x \in \CH_{d+1}(X\times\PP^1\times X\times\PP^1)$ is supported on $X\times \PP^1\times D$ for some divisor $D\subset X\times \PP^1$. If $D\rightarrow \PP^1$ does not hit a general point $y\in\PP^1$, then 
\[
\Gamma = p_{13,*}(p_{13}^*\Gamma \cdot p_2^*x\cdot p_4^*y) = 0,\quad\text{in }\CH_d(X\times X).
\]
If $D\rightarrow \PP^1$ is dominant, then
\[
\Gamma = p_{13,*}(p_{13}^*\Gamma \cdot p_2^*x\cdot p_4^*y)
\]
is supported on $X\times D_y$, where $D_y\subset X$ is a divisor. In either case we have $\gamma = 0$ in $\CH_0(X_K)$ and hence $\varphi_X$ is also injective. It follows that $\varphi_X$ is an isomorphism. To show the compatibility with composition, we note that $x\times \eta_{\PP^1} = \delta_{\PP^1}$ in $\CH_0(\PP^1_{k(\PP^1)})$. For arbitrary $\gamma_1,\gamma_2\in \CH_0(X_K)$, we have
\begin{align*}
\varphi_X(\gamma_1)\circ \varphi_X(\gamma_2) & = (\gamma_1\otimes \delta_{\PP^1})\circ (\gamma_2\otimes\delta_{\PP^1} )\\
 & = (\gamma_1\circ\gamma_2)\otimes (\delta_{\PP^1}\circ \delta_{\PP^1})\\
 & = (\gamma_1\circ\gamma_2)\otimes (\delta_{\PP^1})\\
  & = \varphi_X(\gamma_1\circ \gamma_2).
\end{align*}
We also easily check $\varphi_X(\delta_X) = \delta_X\otimes\delta_{\PP^1} = \delta_{X\times \PP^1}$.

It is clear from (iii) of Corollary \ref{cor generic finite} that the algebra $\CH_0(X_K)$ is a birational invariant. Then (i) implies that it is also a stably birational invariant.
\end{proof}

\section{Birational Chow--K\"unneth decomposition}
In this section we define the notion of a birational Chow--K\"unneth decomposition and study some of its basic properties: stable birational invariance and the induced decomposition of $\CH_0$.

Assume that $k=\C$. Let $\gamma\in \CH_0(X_K)$. Take $\Gamma\in \CH_d(X\times X)$ to be a spreading of $\gamma$, \textit{i.e.} $\Gamma|_{X \times \eta} = \gamma$. Then the homomorphism
\[
[\Gamma]_*: \HH^{i,0}(X) \rightarrow \HH^{i,0}(X),\quad \alpha \mapsto [\Gamma]_*\alpha=p_{2,*}([\Gamma]\cup p_1^*\alpha)
\]
is independent of the choice of the spreading, where $i=0,\ldots,d$. Indeed, if $\Gamma'$ is a different spreading, then the difference $\Gamma'-\Gamma$ is supported on $X \times D$ for some divisor $D\subset X$. Thus the action of this difference is trivial on $\HH^{i,0}(X)$. In this way we get a well-defined \textit{birational cycle class map}
\begin{equation}\label{eq birational cycle class}
[-]: \CH^d(X_K)\rightarrow \bigoplus_{i=0}^d \End(\HH^{i,0}(X)),
\end{equation}
which respects the algebra structures, namely
\[
[\delta_X]=\mathrm{id}\quad \text{and}\quad [\gamma_1\circ \gamma_2] = [\gamma_1]\circ [\gamma_2],\quad \text{for all } \gamma_1,\gamma_2\in \CH_0(X_K).
\]
Let 
$$
\varpi^p_{X,\mathrm{hom}}\in \bigoplus_{i=0}^d \End(\HH^{i,0}(X))
$$ 
be the projector onto $\HH^{p,0}(X)$. Then it is clear that
\[
 [\delta_X] =\mathrm{id} = \sum_{i=0}^d \varpi_{X,\mathrm{hom}}^i.
\]

\begin{defn}\label{defn birational CK}
Let $X$ be a smooth projective variety of dimension $d$ over the field $\C$ of the complex numbers. Let $K=\C(X)$ be function field of $X$. A \textit{birational Chow--K\"unneth decomposition} of $X$ is a decomposition
\begin{equation}\label{eq birational CK}
 \delta_X = \sum_{i=0}^d\varpi_X^i,\quad \text{in }\CH_0(X_{K})
\end{equation}
such that the $\varpi_X^i\in \CH_0(X_K)$ are projectors lifting the cohomological ones in the sense that
\begin{enumerate}[(i)]
\item $\varpi_X^i\circ \varpi_X^j = 0$, for all $i\neq j$;
\item $\varpi_X^i\circ\varpi_X^i = \varpi_X^i$, for all $i=0,1,\ldots,d$;
\item $[\varpi_X^i] = \varpi_{X,\mathrm{hom}}^i$, for all $i=0,1,\ldots,d$.
\end{enumerate}
\end{defn}

\begin{rmk}\label{rmk algebra decomposition}
It is fairly straightforward to see that a birational Chow--K\"unneth decomposition is equivalent to a  decomposition
\[
\CH_0(X_K) = \bigoplus_{i=0}^d A_i
\]
such that (1) each $A_i$ is an algebra with a right unit $\varpi_X^i$ and $\delta_X = \sum \varpi_X^i$ is a decomposition of $\delta_X$ into a sum of projectors; (2) the birational cycle class map $[-]$ restricts to an algebra homomorphism $[-]: A_i\rightarrow \End(\HH^{i,0}(X))$ that respects the units on the two sides. The decomposition of an element is simply given by $\gamma = \sum \gamma\circ\varpi_X^i$; see Proposition \ref{prop algebra decomposition}. We will also call such a decomposition a birational Chow--K\"unneth decomposition. 
\end{rmk}

A similar definition can be made when the base field $k$ is arbitrary. In that case we replace $\HH^{i,0}(X)$ by $\HH^i(X)/\mathrm{N}^1H^i(X)$ where $\HH^i(-)$ is an appropriate Weil cohomology and $\mathrm{N}^1\HH^i(X)$ consists of cohomology classes that are supported on divisors. This is the first quotient associated to the coniveau filtration on cohomology. The Generalized Hodge Conjecture \cite{gro hdg} predicts that the two definitions agree when $k=\C$.

Note that the decomposition in the above definition is with integral coefficients. Hence we do not expect that all varieties $X$ have a birational Chow--K\"unneth decomposition. Note that Conjecture \ref{conj murre} implies that every smooth projective variety admits a birational Chow--K\"unneth decomposition with rational coefficients.

The variety $X$ has a \textit{Chow-theoretical decomposition of the diagonal} if
\[
 \Delta_X = x\times X + \Gamma, \quad\text{in }\CH_d(X\times X)
\]
where $x\in X$ is a closed point and $\Gamma$ is supported on $X\times D$ for some divisor $D$; see \cite{voisin invent}. The birational Chow--K\"unneth decomposition refines this notion.

\begin{prop}
The variety $X$ has a Chow-theoretical decomposition of the diagonal if and only if it has a birational Chow--K\"unneth decomposition of the form $\delta_X = x_K$, for some closed point $x\in X$.\hfill $\square$
\end{prop}

To justify the term ``birational", we have the following.
\begin{prop}\label{prop birational invariance}
Let $X$ and $X'$ be smooth projective varieties over $\C$ that are stably birational to each other. Then $X$ has a birational Chow--K\"unneth decomposition if and only if $X'$ has one.
\end{prop}

\begin{proof}
We first establish the birational invariance. For this, we only need to show that the existence of a birational Chow--K\"unneth decomposition is invariant under a smooth blow up and blow down. Let $X$ be a smooth projective variety and $Y\hookrightarrow X$ be a smooth closed subvariety and let $\rho:\tilde{X}\rightarrow X$ be the blow up of $X$ at center $Y$. We first observe that $\rho^*: \HH^{i,0}(X)\rightarrow \HH^{i,0}(\tilde{X})$ is an isomorphism with inverse given by $\rho_*:  \HH^{i,0}(\tilde{X})\rightarrow \HH^{i,0}({X})$. These induce algebra isomorphisms
\[
[\rho_\#]: \End(\HH^{i,0}(X'))\longrightarrow \End(\HH^{i,0}(X))\quad \text{and}\quad 
[\rho^\#]: \End(\HH^{i,0}(X))\longrightarrow \End(\HH^{i,0}(X')),
\]
which are the cohomological versions of the corresponding maps of Corollary \ref{cor generic finite}. Thus we get a commutative diagram of algebra homomorphisms
\[
\xymatrix{
\CH_0(\tilde{X}_K)\ar[rr]^{[-]}\ar[d]_{\rho_\#} &&\bigoplus_i\End(\HH^{i,0}(\tilde{X}))\ar[d]^{[\rho_\#]} \\
\CH_0({X}_K)\ar[rr]^{[-]} &&\bigoplus_i\End(\HH^{i,0}(\tilde{X}))
}
\]
where the vertical arrows are isomorphisms. It follws that $\CH_0(\tilde{X}_K)$ has a birational Chow--K\"unneth decomposition if and only if $\CH_0(X_K)$ has one.

We move to the invariance under taking product with $\PP^1$. We first note that
\[
 \HH^{i,0}(X\times \PP^1) = \HH^{i,0}(X)\otimes \HH^{0,0}(\PP^1) = \HH^{i,0}(X).
\]
This identification gives rise to a canonical isomorphism of algebras
\[
[\varphi_X]: \End(\HH^{i,0}(X)) \longrightarrow \End(\HH^{i,0}(X\times \PP^1)), 
\]
which is the cohomological version of the  $\varphi_X$ in Proposition \ref{prop stable birational CH0}. Thus we have a commutative diagram
\[
\xymatrix{
\CH_0(X_K)\ar[rr]^{[-]}\ar[d]_{\varphi_X} &&\bigoplus_i\End(\HH^{i,0}(X))\ar[d]^{[\varphi_X]} \\
\CH_0\big( (X\times \PP^1)_{K\otimes \C(\PP^1)} \big)\ar[rr]^{[-]} &&\bigoplus_i\End(\HH^{i,0}(X\times\PP^1))
}
\]
where the vertical arrows are again isomorphisms. Thus $\CH_0(X_K)$ has a Chow--K\"unneth decomposition if and only if $\CH_0\big( (X\times\PP^1)_{K\otimes\C(\PP^1)} \big)$ has one.
\end{proof}

\begin{rmk}\label{rmk explicit CK}
The relation of the projectors can be explicitly written. In the case of blow up, we have $\varpi_{\tilde{X}}^i = \rho^{\#}\varpi_X^i$ and $\varpi_X^i = \rho_\# \varpi_{\tilde{X}}^i$. In the case of taking product with $\PP^1$ we have $\varpi_{X\times\PP^1}^i = \varpi_X^i\otimes \delta_{\PP^1}$.
\end{rmk}

\begin{prop}\label{prop CK of prod}
Let both $X$ and $Y$ be smooth projective varieties with a birational Chow--K\"unneth decomposition. Then $X\times Y$ also admits a birational Chow--K\"unneth decomposition.
\end{prop}

\begin{proof}
The K\"unneth formula implies 
\[
 \HH^{i,0}(X\times Y) = \bigoplus_{p+q=i}\HH^{p,0}(X)\otimes \HH^{q,0}(Y).
\]
Given birational Chow--K\"unneth decompositions $\varpi_X^p$, $0\leq p \leq d_X$, and $\varpi_Y^q$, $0\leq q \leq d_Y$, of $X$ and $Y$ respectively. Then
\[
 \varpi_{X\times Y}^i : = \sum_{p+q=i}\varpi_X^p\otimes\varpi_Y^q,\quad 0\leq i\leq d_X+d_Y,
\]
form a birational Chow--K\"unneth decomposition of $X\times Y$.
\end{proof}

Note that any element $\gamma\in \CH_0(X_K)$ defines a homomorphism
\[
\gamma^*: \CH_0(X) \rightarrow \CH_0(X), \qquad \alpha\mapsto \gamma^*\alpha := p_{1,*}(\Gamma\cdot p_2^*\alpha),
\]
where $\Gamma\in \CH_d(X\times X)$ is a spreading of $\gamma$. As a consequence, associated to a birational Chow--K\"unneth decomposition \eqref{eq birational CK}, we have a decomposition
\[
 \CH_0(X) = \bigoplus_{i=0}^d \CH_0(X)_i,\qquad \text{where }\CH_0(X)_i := (\varpi_X^i)^*\CH_0(X). 
\]
Moreover, this construction can be carried out universally in the following sense. Let $L$ be a field extension of $\C$. The diagonal $\Delta_{X_L}\subset X_L\times_L X_L$ gives rise to the element $\delta_{X_L}\in \CH_0(X_{K\otimes_{\C} L})$. By base change the equation \eqref{eq birational CK}, we have
\[
 \delta_{X_L} = \sum_{i=0}^d \varpi_{X_L}^i,\qquad \text{where }\varpi_{X_L}^i = \varpi_X^i \otimes_{\C}L. 
\]
Then we have the decomposition
\begin{equation}\label{eq universal decomp}
\CH_0(X_L) = \bigoplus_{i=0}^d \CH_0(X_L)_i, \quad\text{where } \CH_0(X_L)_i:=(\varpi_{X_L}^i)^*\CH_0(X_L).
\end{equation}

\begin{prop}\label{prop algebra decomposition}
Under the decomposition as above, we have
\[
 \varpi_X^i \in \CH_0(X_K)_i
\]
for all $i=0,\ldots,d$. Each group $\CH_0(X_K)_i$
is an algebra with a right unit and the birational cycle class map restricts to a homomorphism of algebras $[-]: \CH_0(X_K)_i\rightarrow \End(\HH^{i,0}(X))$.
\end{prop}

\begin{proof}
Let $\Pi_X^i\subset X\times X$ be a spreading of $\varpi_X^i$. The $\Pi_X^i\otimes K\subset X\times X\times \Spec K$ is a spreading of $\varpi^i_{X_K}$. Then we have
\begin{align*}
( \varpi^j_{X_K})^*\varpi^i_{X} & = p_{13,*}(p_{12}^*\Pi_X^j\cdot p_{23}^*\Pi^i_X)|_{X\times \eta_X}\\
 & = (\varpi_X^i\circ\varpi_X^j)\\
 & =
 \begin{cases}
 0, &\text{if }i\neq j\\
 \varpi_{X}^i, & \text{if }i=j.
 \end{cases}
\end{align*}
This exactly means that $\varpi_X^i \in \CH_0(X_K)_i$, $0\leq i\leq d$. We need to show that $\CH_0(X_K)_i$ inherits the structure of an algebra with $\varpi_X^i$ being a right unit. Actually, the above computation also shows that $(\varpi^j_{X_K})^*\gamma = \gamma\circ \varpi_X^j$ for all $\gamma\in \CH_0(X_K)$. Let $\gamma\circ\varpi_X^i$ and $\gamma'\circ\varpi_X^i$ be two elements in $\CH_0(X_K)_i$, then
\[
(\gamma\circ\varpi_X^i)\circ(\gamma'\circ\varpi_X^i) = (\gamma\circ\varpi_X^i\circ\gamma')\circ\varpi_X^i \in \CH_0(X_K)_i.
\]
This shows that $\CH_0(X_K)_i$ is a sub-algebra with right unit $\varpi_X^i$.
\end{proof}

\section{Jacobian varieties}

In this section we show that each Jacobian variety admits a birational Chow-K\"unneth decomposition. This is done by first proving that for the symmetric products of a curve and then using the stably birational invariance of the existence of a birational Chow--K\"unneth decomposition.

Let $C$ be a smooth projective curve of genus $g$. Let $n$ be a positive integer. We use $C^{[n]}$ to denote the Hilbert scheme of length-$n$ subschemes of $C$, which is the same as the $n$-fold symmetric product of $C$. For each $\xi\in C^{[n]}$, we use $Z_{\xi}\subset C$ to denote the corresponding close subscheme. For any positive integer $r$, let 
\[
C^{[n-r,n]} =\Set{(\xi_1,\xi_2)\in C^{[n-r]} \times C^{[n]}}{Z_{\xi_1}\subset Z_{\xi_2}},
 \] 
 viewed as a correspondence from $C^{[n-r]}$ to $C^{[n]}$. Its transpose will be denoted $C^{[n,n-r]}$. Using the notation of \cite{rennemo}, we write
 \[
  \mu_{+}(C,n) = C^{[n,n+1]}\in \CH_{n+1}(C^{[n]} \times C^{[n+1]})\quad \text{and}\quad \mu_{-}(C,n) = C^{[n,n-1]}\in \CH_n(C^{[n]} \times C^{[n-1]}).
 \]
 After fixing a point $x_0\in C$, we can define a morphism
 \[
  \iota_n: C^{[n]}\longrightarrow C^{[n+1]},\qquad \xi \mapsto \xi+x_0.
 \]
 Such a morphism induces two other correspondences
 \[
  \mu_+(x_0, n) = (\iota_n)_* \quad \text{and} \quad \mu_{-}(x_0,n) = (\iota_{n-1})^*.
 \]
In practice, we simply write $\mu_{\pm}(C)$ and $\mu_{\pm}(x_0)$ with the correct dimension being understood. 
\begin{prop}[Rennemo \cite{rennemo}]
The correspondences $\mu_{\pm}(x_0)$ and $\mu_{\pm}(C)$ satisfy the commutation relations
\begin{align*}
[\mu_{-}(x_0),\mu_{+}(C)] & = \Delta_{C^{[n]}},\\
[\mu_{-}(C), \mu_{+}(x_0)] & = \Delta_{C^{[n]}}
\end{align*}
and all other possible commutators vanish.\hfill $\square$
\end{prop}

Let $\iota_r: C^{[n-r]}\rightarrow C^{[n]}$ be the morphism defined by $\iota_r(\xi)=\xi + r x_0$. Then we define
\[
\Gamma_r:=(\iota_r\times \mathrm{id})_* C^{[n-r,n]} \in \CH_n(C^{[n]} \times C^{[n]})
\]
and
\[
\gamma_r : = \Gamma_r|_{C^{[n]}\times \eta_{C^{[n]}}} \in \CH_0\big(C^{[n]}\times \eta_{C^{[n]}}\big).
\]
Let $\Gamma_0 := \Delta_{C^{[n]}}$ and $\Gamma_r=0$ for all $r>n$. We formally write
\[
 \mathcal{A}_n := \bigoplus_{r=0}^n \Z\, \Gamma_r 
\]
and 
\[
\mathcal{A}'_n : =\bigoplus_{r=0}^n \Z\, \gamma_r,
\]
where $\gamma_0=\delta_{C^{[n]}}$.

\begin{lem}\label{lem Cn}
Assume that $1\leq n\leq g$. Then the following statements are true.
\begin{enumerate}[(i)]
\item The natural homomorphism $\mathcal{A}_n\rightarrow \CH_n(C^{[n]} \times C^{[n]})$ is injective. Similarly, the homomorphism $\mathcal{A}'_n\rightarrow \CH_0(C^{[n]}\times \eta_{C^{[n]}})$ is also injective.

\item For any non-negative integer $r$, we have $\mu_{+}(C)^r\circ \mu_{-}(x_0)^r = r!\, \Gamma_r$ in $\CH_n(C^{[n]} \times C^{[n]})$.

\item Let $l$ and $m$ be two non-negative integers, then
\[
\mu_{-}(x_0)^l \circ \mu_{+}(C)^m = \sum_{i=0}^l {l \choose i} m(m-1)\cdots(m-i+1) \mu_{+}(C)^{m-i}\circ \mu_{-}(x_0)^{l-i},
\]
where the coefficient is understood to be 1 for the term $i=0$ in the summation.

\item For non-negative integers $r$ and $r'$, we have
\[
 \Gamma_r\circ\Gamma_{r'} = \sum_{i=0}^{\min\{r,r'\}} \frac{(r+r'-i)!}{i! (r-i)!(r'-i)!} \Gamma_{r+r'-i}
\]
in $\CH_n(C^{[n]}\times C^{[n]})$.

\item $\mathcal{A}_n$ is a commutative sub-algebra of $\CH_n(C^{[n]} \times C^{[n]})$. Similarly, $\mathcal{A}'_n$ a commutative sub-algebra of $\CH_0(C^{[n]}\times \eta_{C^{[n]}})$.
\end{enumerate}
\end{lem}

\begin{proof}
For (i), we first note that the cycles $\Gamma_r$, $-n\leq r \leq n$, are $\Z$-linearly independent. Indeed, they are independent in cohomology, which can be easily checked by pulling back to $C^n\times C^n$. The same argument works for $\mathcal{A}'_n$.

We show that $\mu_{-}(C)^r = r!\, C^{[n,n-r]}$ in $\CH_n(C^{[n]} \times C^{[n-r]})$. Indeed, the computation at the level of cycles shows that $\mu_{-}(C)^r$ is a multiple of $C^{[n,n-r]}$. Over a generic point of $C^{[n]}$, the action of $\mu_{-}(C)$ is ``removing one point", by which we mean
\[
\xi = y_1 + y_2 + \cdots +y_n \mapsto \sum_{i=1}^n [\xi - y_i]
\]
where $[\xi - y_i]$ is viewed as a point on $C^{[n-1]}$. Thus the action of $\mu_{-}(C)^r$ is ``removing $r$ points one after another". The correspondence $C^{[n,n-r]}$ acts generically as ``removing $r$ points" without remembering the order. Hence we get that the coefficient in front of $C^{[n,n-r]}$ is $r!$. By taking the composition with $\mu_+(x_0)^r$, we have
\[
 \mu_+(x_0)^r\circ \mu_{-}(C)^r =r!\, (\mathrm{id}_{C^{[n]}} \times \iota_r)_*C^{[n,n-r]}=r!\,{}^t\Gamma_r.
\]
Statement (ii) follows by taking the transpose of the above equation.

To prove (iii), we take the model $\mu_{-}(x_0) = \frac{\mathrm{d}}{\mathrm{d}x}$ and $\mu_+(C) = x$ as operators on $\Z[x]$. The formula can be deduced easily by induction.

Use the notation of operators on $\Z[x]$, we have
\begin{align*}
 r!r'! \Gamma_r\circ \Gamma_{r'} & = x^r\left(\frac{\mathrm{d}}{\mathrm{d}x}\right)^r x^{r'} \left(\frac{\mathrm{d}}{\mathrm{d}x}\right)^{r'}\\
 & = x^r\sum_{i=0}^r{r \choose i} r'(r'-1)\cdots(r'-i+1) x^{r'-i} \left(\frac{\mathrm{d}}{\mathrm{d}x}\right)^{r-i} \left(\frac{\mathrm{d}}{\mathrm{d}x}\right)^{r'} \\
 & = \sum_{i=0}^{\min\{r,r'\}}\frac{r!}{i!(r-i)!} \frac{r'!}{(r'-i)!} x^{r+r'-i}\left(\frac{\mathrm{d}}{\mathrm{d}x}\right)^{r+r'-i}\\
 & = r!r'! \sum_{i=0}^{\min\{r,r'\}}\frac{(r+r'-i)!}{i!(r-i)!(r'-i)!}\Gamma_{r+r'-i}.
\end{align*}
By (i), we can cancel the factor $r!r'!$ and get (iv)

Statement (v) follows from (i) and (iv).
\end{proof}

\begin{lem}\label{lem CK for Cn}
For every integer $1\leq n\leq g$, the symmetric product $C^{[n]}$ admits a birational Chow--K\"unneth decomposition.
\end{lem}

\begin{proof}
Let $f: C^n \rightarrow C^{[n]}$ be the quotient map and let
\[
f\times f: C^n\times C^n \longrightarrow C^{[n]} \times C^{[n]}
\]
be the two-fold self product of $f$. Then we have
\[
 (f\times f)_*\Delta_{C^n} = n!\, \Delta_{C^{[n]}}. 
\]
After fixing a closed point $x_0\in C$, we have a Chow--K\"unneth decomposition of $C$ as follows
\[
\pi_C^0 = x_0\times C, \qquad \pi_C^2 = C\times x_0,\qquad \pi_C^1 = \Delta_C - \pi_C^0 -\pi_C^2.
\]
This gives rise to a Chow--K\"unneth decomposition of $C^n$ given by
\[
\pi_{C^n}^ i =\sum_{i_1+\cdots + i_n =i} \pi_C^{i_1}\otimes\cdots \otimes \pi_{C}^{i_n}, \qquad 0\leq i \leq 2n.
\]
If $i_r=2$, for some $0\leq r\leq n$, then the restriction of $(f\times f)_*(\pi_C^{i_1}\otimes \cdots \otimes \pi_C^{i_n})$ to $C^{[n]}\times \eta_{C^{[n]}}$ is trivial. When $i>n$, the condition $i_1+\cdots +i_n=i$ will force $i_r=2$ for some $r$. It follows that 
\[
(f\times f)_*\pi_{C^n}^i|_{C^{[n]}\times \eta_{C^{[n]}}} = 0, \quad\text{in }\CH_0(C^{[n]}\times \eta_{C^{[n]}}),
\]
for all $n+1\leq i\leq 2n$. By restricting the cycles to $C^{[n]}\times \eta_{C^{[n]}}$, we see that
\[
(f\times f)_* \pi^i_{C^n} = n!\, \Pi_i, \quad\text{in }\CH_0(C^{[n]}\times \eta_{C^{[n]}}),\quad 0\leq i\leq n,
\]
for some $\Pi_i\in \mathcal{A}_n$. Apply $(f\times f)_*$ to the Chow--K\"unneth decomposition of $C^n$ and we get
\[
 n! \delta_{C^{[n]}} = n! \sum_{i=0}^n \varpi^i_{C^{[n]}},\quad \text{in }\CH_0(C^{[n]}\times \eta_{C^{[n]}}),
\]
where $\varpi_{C^{[n]}}^i = \Pi_i|_{C^{[n]}\times \eta_{C^{[n]}}}\in \mathcal{A}'_n$. By (i) of Lemma \ref{lem Cn}, we can cancel the factor $n!$ and get the equality
\[
 \delta_{C^{[n]}} = \sum_{i=0}^n \varpi^i_{C^{[n]}},\quad \text{in }\CH_0(C^{[n]}\times \eta_{C^{[n]}}).
\]
We check that the action of $\varpi_{C^{[n]}}^i$ on cohomology is exactly the required one for a birational Chow--K\"unneth decomposition. This is done by checking its action on the $S_n$-invariant part of the cohomology of $C^{n}$.
\end{proof}

\begin{thm}\label{thm jac}
Every Jacobian variety admits a birational Chow--K\"unneth decomposition.
\end{thm}
\begin{proof}
Let $C$ be a smooth projective curve of genus $g$. Then its Jacobian $\mathrm{Jac}(C)$ is birational to $C^{[g]}$. The theorem follows from Lemma \ref{lem CK for Cn} together with the birational invariance of the existence of a birational Chow--K\"unneth decomposition (Proposition \ref{prop birational invariance}). 
\end{proof}

\begin{cor}
For all positive integer $n$, the symmetric product $C^{[n]}$ has a birational Chow--K\"unneth decomposition.
\end{cor}
\begin{proof}
Lemma \ref{lem CK for Cn} takes care of $n\leq g$. When $n>g$, we know that $C^{[n]}\rightarrow \mathrm{Jac}(C)$ is generically a projective bundle. Hence $C^{[n]}$ and $\mathrm{Jac}(C)$ are stably birational to each other.
\end{proof}

\section{the Hilbert scheme of a $K3$ surface}

Let $S$ be an algebraic complex $K3$ surface. In this section we show that the Hilbert scheme of points on $S$ admits a birational Chow--K\"unneth decomposition.

\begin{thm}[Beauville--Voisin \cite{bv}]\label{thm bv}
There is a canonical class $\mathfrak{o} = \mathfrak{o}_S$ in $\CH_0(S)$, which is represented by any point on a rational curve of $S$ and satisfies the following conditions.
\begin{enumerate}[(i)]
\item Let $D_1$ and $D_2$ be two divisors on $S$, then $D_1\cdot D_2 = \deg(D_1\cdot D_2)\mathfrak{o}_S$ in $\CH_0(S)$.
\item The second Chern class of $S$ is given by $c_2(S)=24 \mathfrak{o}_S$ in $\CH_0(S)$.
\end{enumerate}
\end{thm}

Let $S^{[n]}$ be the Hilbert scheme of length-$n$ closed subschemes of $S$ and we know that $S^{[n]}$ is a compact hyperk\"ahler manifold. It is shown by Vial \cite{vial} that $S^{[n]}$ admits a Chow--K\"unneth decomposition that is multiplicative. The result of Vial is with rational coefficients, though it is likely that the same result holds for integral coefficients. When we restrict to the birationally invariant part of the motive, it is much easier to see that a birational Chow--K\"unneth decomposition exists on $S^{[n]}$.

The method is very similar to that of the case $C^{[n]}$. Let $S^{[n-r,n]}\subset S^{[n-r]}\times S^{[n]}$ be the main component of 
\[
\Set{(\xi_1,\xi_2)\in S^{[n-r]} \times S^{[n]}}{ Z_{\xi_1}\subseteq Z_{\xi_2}}
\]
which dominates the factor $S^{[n]}$. We view $S^{[n-r,n]}$ as an element in $\CH_{2n}(S^{[n-r]} \times S^{[n]})$. Let $S^{[n,n-r]}\in \CH_{2n}(S^{[n]} \times S^{[n-r]})$ be the transpose of $S^{[n-r,n]}$. This allows us to define
\[
\mu_{-}(S) = S^{[n,n-1]} \in \CH_{2n}(S^{[n]} \times S^{[n-1]})\quad\text{and} \quad \mu_{+}(S) = S^{[n-1,n]} \in \CH_{2n}(S^{[n-1]} \times S^{[n]})
\]
Fix distinct points $x_1,x_2,\ldots, x_l, \ldots \in S$ that are of the class $\mathfrak{o}_S$. Such a choice gives rise to rational maps
\[
 \iota_r : S^{[n-r]}\dashrightarrow S^{[n]},\quad \xi \mapsto \xi + x_1 +\cdots + x_r
\]
where $\xi$ is supported away from $x_1,\ldots,x_r$. Let $\Gamma_{\iota_r}\subset S^{[n-r]}\times S^{[n]}$ be the closure of the graph of $\iota_r$.

\begin{defn}
Let $\mu_{-}(\mathfrak{o})\in \CH_0(S^{[n]} \times \eta_{S^{[n-1]}})$ be the restriction of the correspondence ${}^t\Gamma_{\iota_1}\in \CH_{2n-2}(S^{[n]} \times S^{[n-1]})$. Let $\mu_{+}(\eta) \in \CH_0(S^{[n-1]} \times \eta_{S^{[n]}})$ be the restriction of $\mu_+(S)$.
\end{defn}

\begin{rmk}
The definition of $\mu_{-}(\mathfrak{o})$ is independent of the choice of $x_1$. We use the same notation when $n$ varies; if we compose such correspondences, the correct dimension is being understood.
\end{rmk}

\begin{prop}
The above correspondences satisfy $[\mu_{-}(\mathfrak{o}), \mu_{+}(\eta)] = 1$ in $\CH_0(S^{[n]}\times \eta_{S^{[n]}})$.
\end{prop}

\begin{proof}
Let $\xi= [y_1,\ldots,y_n]\in S^{[n]}$ be a general point. The cycle $\Gamma' :=p_{13,*}\Big( p_{12}^*S^{[n,n+1]} \cdot p_{23}^*{}^t\Gamma_{\iota_1} \Big)$ meets $S^{[n]} \times \xi$ in the points
\[
([x_1,y_1,\ldots,\hat{y}_i, \ldots,y_n],\xi),\,\, 1\leq i\leq n\quad \text{and}\quad (\xi,\xi).
\]
The cycle $\Gamma'':= p_{13,*}\Big( p_{12}^*{}^t\Gamma_{\iota_1} \cdot p_{23}^*S^{[n-1,n]} \Big)$ meets $S^{[n]} \times \xi$ in the points
\[
 ([x_1,y_1,\ldots,\hat{y}_i, \ldots,y_n],\xi),\,\, 1\leq i\leq n.
\]
It follows that $\Gamma' -\Gamma'' -\Delta_{S^{[n]}} \in \CH_{2n}(S^{[n]} \times S^{[n]})$ is supported on $S^{[n]}\times D$ for some divisor $D\subset S^{[n]}$. Note that 
\[
\Gamma'|_{S^{[n]} \times \eta_{S^{[n]}}} = \mu_{-}(\mathfrak{o})\circ \mu_{+}(\eta),\qquad \Gamma''|_{S^{[n]} \times \eta_{S^{[n]}}} = \mu_{+}(\eta)\circ \mu_{-}(\mathfrak{o}).
\]
Since the restriction of $\Gamma'-\Gamma''-\Delta_{S^{[n]}}$ to $S^{[n]} \times \eta_{S^{[n]}}$ is zero, we get $[\mu_{-}(\mathfrak{o}),\mu_{+}(\eta)] = 1$.
\end{proof}

As in the case of $C^{[n]}$, we need to divide $\mu_{+}(\eta)^r\circ \mu_{-}(\mathfrak{o})^r$ by $r!$. For that, we take
\[
\Gamma_r := (\iota_r\times \mathrm{id})_* S^{[n-r,n]} \in \CH_{2n}(S^{[n]} \times S^{[n]}).
\]
By convention, we write $\Gamma_0 := \Delta_{S^{[n]}}$. Let $\gamma_r\in\CH_0(S^{[n]} \times \eta_{S^{[n]}})$ be the restriction of $\Gamma_r$ to $S^{[n]} \times \eta_{S^{[n]}}$. We define
\[
 \mathcal{A}'_n := \bigoplus_{r=0}^{n} \Z\,\gamma_r.
\]

\begin{lem}
For all positive integer $n$, the following statements are true.
\begin{enumerate}[(i)]
\item For all $0\leq r \leq n$, we have $r!\,\gamma_r = \mu_{+}(\eta)^r\circ \mu_{-}(\mathfrak{o})^r$ in $\CH_0(S^{[n]} \times \eta_{S^{[n]}})$.
\item The natural homomorphism $\mathcal{A}'_n\rightarrow \CH_0(S^{[n]} \times \eta_{S^{[n]}})$ is injective.
\item For all $0\leq r,r' \leq n$, we have
\[
 \gamma_{r}\circ \gamma_{r'} = \sum_{i=0}^{\min\{r,r'\}} \frac{(r+r'-i)!}{i!(r-i)!(r'-i)!}\, \gamma_{r+r'-i}
\]
in $\CH_0((S^{[n]} \times \eta_{S^{[n]}})$, where $\gamma_l=0$ for all $l>n$.
\item $\mathcal{A}'_n$ is a commutative sub-algebra of $\CH_0((S^{[n]} \times \eta_{S^{[n]}})$.
\end{enumerate}
\end{lem}

\begin{proof}
The proof is the same as in the case of $C^{[n]}$. Let $\xi=[y_1,\ldots,y_n]\in S^{[n]}$ be a general point. At the level of cycles, the intersection of $S^{[n]} \times \xi$ and a cycle representing $\mu_{+}(S)^{r}\circ \iota_r^*$ is supported on
\[
 \Set{[x_1,\ldots,x_r, y_{i_1}, \ldots, y_{i_{n-r}}]}{ 0\leq i_1< \cdots < i_{n-r}\leq n}
\]
Thus $\mu_{+}(\eta)^r\circ \mu_{-}(\mathfrak{o})^r$ is a multiple of $\gamma_r$. Then we argue as in the case of $C^{[n]}$ that the coefficient is $r!$. This establishes statement (i).

To prove (ii), we compose the homomorphism in question with the birational cycle class map. It turns out that the resulting map $\mathcal{A}'_n \rightarrow \oplus_{i=0}^{2n} \End(\HH^{i,0}(X))$ is injective.

The proof of (iii) and (iv) is the same as in Lemma \ref{lem Cn}.
\end{proof}

\begin{thm}\label{thm Hilb}
Let $S$ be a complex algebraic $K3$ surface. Then $S^{[n]}$ admits a canonical birational Chow--K\"unneth decomposition for all $n>0$.
\end{thm}

\begin{proof}
Let $f: S^n\dashrightarrow S^{[n]}$ be the quotient rational map. The closure $\bar{\Gamma}_f$ of the graph of $f$ restricts to a well-defined element $\gamma_f\in \CH_0(S^n\times \eta_{S^{[n]}})$. The transpose ${}^t\bar{\Gamma}_f$ defines an element ${}^t\gamma_f\in \CH_0(S^{[n]} \times \eta_{S^n})$. Then we have $f_{\#}\delta_{S^n}=\gamma_f\circ {}^t\gamma_f = n!\delta_{S^{[n]}}$; see Corollary \ref{cor generic finite}. Let
\[
 \pi_S^0 = \mathfrak{o}_S\times S,\quad\pi^1_S=\pi_S^3=0,\quad \pi_S^4 = S\times \mathfrak{o}_S,\quad \pi_S^2 = \Delta_S - \pi^0_S-\pi^4_S
\]
be the canonical Chow--K\"unneth decomposition of $S$. It gives rise to a product Chow--K\"unneth decomposition of $S^n$ given by
\[
\pi^j_{S^n} :=\sum_{i_1+\cdots+i_n=j}\pi_S^{i_1} \otimes \cdots \otimes \pi_S^{i_n}, \quad 0\leq j\leq 4n
\]
When we restrict to $S^{n}\times \eta_{S^n}$, we get the birational Chow--K\"unneth decomposition
\[
 \varpi_{S^n}^j = \sum_{i_1+\cdots+i_n=j} \varpi_S^{i_1}\otimes \cdots \varpi_S^{i_n},\quad 0\leq j\leq 2n.
\]
We note, as in the case of $C^{[n]}$, that $f_{\#}\varpi_{S^n}^j$ is a linear conbination of $\gamma_r\in \mathcal{A}'_n$ with coefficients divisible by $n!$. Hence we can write
\[
 f_\# \varpi_{S^n}^j = n! \varpi_{S^{[n]}}^j
\]
for some $\varpi_{S^{[n]}}^j\in \mathcal{A}'_n$. It is easy to see that the action of $\varpi_{S^{[n]}}^i$ on the cohomology groups $\HH^{i,0}(S^{[n]})$ is exactly the expected one. Since $\mathcal{A}'_n$ injects into cohomology, we see that $\varpi_{S^{[n]}}^j$ are actual projectors and
\[
\delta_{S^{[n]}} = \sum_{j=0}^{2n}\varpi_{S^{[n]}}^j,\quad\text{in }\CH_0(S^{[n]} \times \eta_{S^{[n]}}).
\]
Thus these projectors form a birational Chow--K\"unneth decomposition.
\end{proof}

\begin{rmk}
The construction also gives that $\varpi_{S^{[n]}}^j=0$ for all odd $j$. The induced decomposition of $\CH_0(S^{[n]})$ is the same as the one considered by Voisin \cite{voisin isotropy}.
\end{rmk}

\section{The case of $S^{[2]}$}
In the case of $F=S^{[2]}$, we show that an integral Chow--K\"unneth decomposition exists.

Let $S$ be an algebraic complex $K3$ surface. By Theorem \ref{thm bv} we know that $\CH_0(S)$ contains a canonical degree one element $\mathfrak{o}=\mathfrak{o}_S$, which is the class of a point on any rational curve. For a point $x\in S$, there is an associated smooth surface $S_x\subset S^{[2]}$ which is the closure of all points $[x,y]$, where $y\in S-\{x\}$. We also use $S_x$ to denote its class in $\CH_2(S^{[2]})$. Hence $\CH_2(S^{[2]})$ contains a canonical class $S_{\mathfrak{o}}$. Let
\[
I=\Set{(Z_1,Z_2)\in S^{[2]}\times S^{[2]}}{Z_1\cap Z_2 \neq \emptyset}
\]
be the incidence correspondence. There is an induced canonical class $\mathfrak{o}_F = [\mathfrak{o}, \mathfrak{o}]$ of degree one in $\CH_0(S^{[2]})$. Consider the diagram
\[
\xymatrix{
 E\ar[r]^{j\quad}\ar[d]_{\pi} &\widetilde{S\times S}\ar[d]^\rho \ar[r]^\sigma &S^{[2]}\\
 S\ar[r]^{\iota\quad} &S\times S &
}
\]
where $\iota$ is the diagonal embedding, $\rho$ is the blow-up along diagonal and $E$ is the exceptional divisor. This allows us to define a canonical element
\[
E_{\mathfrak{o}} =\sigma_* j_*\pi^*\mathfrak{o}_S \in \CH_1(S^{[2]}).
\]
 Let $\delta\in \CH^1(F)$ be the half diagonal, namely $\sigma^*\delta = j_*E$. The following is a result of \cite[Chapter 13]{shen-vial} made explicit.

\begin{prop}\label{prop CK S2}
The Hyperk\"ahler variety $F=S^{[2]}$ admits an integral Chow--K\"unneth decomposition
\begin{equation}\label{eq CK S2}
 \Delta_F = \pi^0_F + \pi_F^2 + \pi_F^4 +\pi_F^6 + \pi_F^8,
\end{equation}
where
\begin{align*}
 \pi_F^0 & = \mathfrak{o}_F\times F,\\
 \pi_F^2 & = I\cdot p_1^*S_{\mathfrak{o}} - 2\mathfrak{o}_F \times F - 2S_{\mathfrak{o}} \times S_{\mathfrak{o}} - E_{\mathfrak{o}} \times \delta,\\
 \pi_F^8 & = F \times \mathfrak{o}_F,\\
 \pi_F^6 &= {}^t\pi_F^2,\\
 \pi_F^4 &= \Delta_X - \pi_F^0 -\pi_F^2 -\pi_F^6 -\pi_F^8.
\end{align*}
\end{prop}

\begin{cor}
The hyperk\"ahler variety $S^{[2]}$ admits a birational Chow--K\"unneth decomposition.
\end{cor}

\begin{proof}[Proof of Proposition \ref{prop CK S2}]
Let $\mathfrak{a}_i$, $i=1,2,\ldots, 22$, be an integral basis of $\HH^2(S,\Z)$. We write 
\[
\hat{\mathfrak{a}}_i = \sigma_*\rho^*(\mathfrak{a}_i \otimes 1) = Z^*\mathfrak{a}_i,\quad I=1,\ldots,22,
\]
where $Z\cong \widetilde{S\times S}$ is the universal family over $F$ viewed as a correspondence between $F$ and $S$. Then
$\{\delta,\hat{\mathfrak{a}}_1,\ldots,\hat{\mathfrak{a}}_{22}\}$ form an integral basis of $\HH^2(F,\Z)$. We view $Z\times Z$ as a correspondence between $F\times F$ and $S\times S$. Then
\[
 I = (Z\times Z)^* \Delta_S,\quad \text{ in }\CH^2(F\times F).
\]
Note that $S$ admits an integral multiplicative Chow--K\"unneth decomposition
\[
 \Delta_S = \mathfrak{o}\times S + \pi_S^2 + S\times \mathfrak{o}.
\]
It follows that
\[
 I = 2S_{\mathfrak{o}} \times F + (Z\times Z)^*\pi_S^2 +2 F\times S_{\mathfrak{o}}.
\]
Let $A=(a_{ij})$ be the intersection matrix of $\HH^2(S,\Z)$ with respect to the given basis. Let $B=(b_{ij})=A^{-1}$. Then the cohomomolgy class of $\pi_S^2$ is given by $\sum b_{ij}\mathfrak{a}_i\otimes \mathfrak{a}_j$. As a result, we have
\[
[I] =2 [S_{\mathfrak{o}}]\otimes [F] + \sum b_{ij}\hat{\mathfrak{a}}_i\otimes \hat{\mathfrak{a}}_j +2 [F]\otimes [S_{\mathfrak{o}}],\quad \text{ in }\HH^4(F\times F).
\]
A direct computation gives
\[
I\cdot p_1^*S_{\mathfrak{o}} = 2 \mathfrak{o}_F \times F + p_1^*S_{\mathfrak{o}} \cdot (Z\times Z)^*\pi_S^2 + 2 S_{\mathfrak{o}} \times S_{\mathfrak{o}}.
\]
Here we use the fact that $S_x\cdot S_y =[x,y]$, and in particular $S_\mathfrak{o}\cdot S_{\mathfrak{o}}= \mathfrak{o}_F$. The cohomology class of $S_\mathfrak{o}$ satisfies the following conditions
\[
[S_{\mathfrak{o}}]\cdot\hat{\mathfrak{a}}_i\cdot\hat{\mathfrak{a}}_j = a_{ij},\quad [S_{\mathfrak{o}}]\cdot\hat{\mathfrak{a}}_i\cdot \delta = 0,
\quad [S_{\mathfrak{o}}]\cdot\delta\cdot\delta=-1.
\]
It follows that $\sum_j b_{ij} [S_{\mathfrak{o}}]\cdot \hat{\mathfrak{a}}_j = \hat{\mathfrak{a}}_i^\vee$ and $-[E_\mathfrak{o}] = \delta^\vee$. Thus the cohomology class of $p_1^*S_{\mathfrak{o}}\cdot (Z\times Z)^*\pi_S^2$ is $\sum_i \hat{\mathfrak{a}}_i^\vee \otimes \hat{\mathfrak{a}}_i$. Hence the cohomology class of
\[
 \pi^2_F = I\cdot p_1^*S_{\mathfrak{o}} -2 \mathfrak{o}_F \times F - 2S_{\mathfrak{o}} \times S_{\mathfrak{o}} - E_{\mathfrak{o}} \times \delta
\]
is the cohomological projector $\pi_{F,\mathrm{hom}}^2$. One similarly checks that all the $\pi_F^i$ lifts the corresponding cohomological projectors. 

We still need to show that these are projectors as correspondences. This can be checked directly. For example, we have
\begin{align*}
 \big( I\cdot p_1^*S_{\mathfrak{o}} \big) \circ \big( I\cdot p_1^*S_{\mathfrak{o}} \big) = I\cdot p_1^*S_{\mathfrak{o}} + 2\mathfrak{o}_F\times F + 2 S_{\mathfrak{o}} \times S_{\mathfrak{o}}, \qquad  
 & \big( I\cdot p_1^*S_{\mathfrak{o}} \big) \circ  \big( \mathfrak{o}_F\times F \big) = 2\mathfrak{o}_F\times F,\\
   \big( I\cdot p_1^*S_{\mathfrak{o}} \big) \circ \big( S_{\mathfrak{o}} \times S_{\mathfrak{o}} \big) = 2  S_{\mathfrak{o}} \times S_{\mathfrak{o}},\qquad 
   & \big( I\cdot p_1^*S_{\mathfrak{o}} \big) \circ \big( E_{\mathfrak{o}} \times \delta \big) = 0, \\   
 \big( \mathfrak{o}_F\times F \big) \circ \big( I\cdot p_1^*S_{\mathfrak{o}} \big) =  2 \mathfrak{o}_F\times F,\qquad & \big( \mathfrak{o}_F\times F \big) \circ \big( \mathfrak{o}_F\times F \big) = \mathfrak{o}_F\times F,\\
 \big( \mathfrak{o}_F\times F \big) \circ \big(S_{\mathfrak{o}} \times S_{\mathfrak{o}} \big) = 0,\qquad 
 &\big( \mathfrak{o}_F\times F \big) \circ \big( E_{\mathfrak{o}} \times \delta \big) = 0,\\
  \big(S_{\mathfrak{o}} \times S_{\mathfrak{o}} \big) \circ (I\cdot p_1^*S_{\mathfrak{o}}) = 2 S_{\mathfrak{o}} \times S_{\mathfrak{o}},\qquad & \big(S_{\mathfrak{o}} \times S_{\mathfrak{o}} \big) \circ \big( \mathfrak{o}_F\times F \big) = 0, \\
  \big(S_{\mathfrak{o}} \times S_{\mathfrak{o}} \big) \circ \big(S_{\mathfrak{o}} \times S_{\mathfrak{o}} \big) = S_{\mathfrak{o}} \times S_{\mathfrak{o}} ,\qquad 
  & \big(S_{\mathfrak{o}} \times S_{\mathfrak{o}} \big) \circ \big( E_{\mathfrak{o}} \times \delta \big) = 0,\\
  \big( E_{\mathfrak{o}} \times \delta \big) \circ \big( I\cdot p_1^*S_{\mathfrak{o}} \big) = 0,\qquad
  & \big( E_{\mathfrak{o}} \times \delta \big) \circ  \big( \mathfrak{o}_F\times F \big) =0,\\
   \big( E_{\mathfrak{o}} \times \delta \big) \circ  \big(S_{\mathfrak{o}} \times S_{\mathfrak{o}} \big) =0, \qquad & \big( E_{\mathfrak{o}} \times \delta \big) \circ \big( E_{\mathfrak{o}} \times \delta \big) = - E_{\mathfrak{o}} \times \delta.
\end{align*}
From the above identities, we get $\pi_F^i\circ \pi^j_F =0$, for all $i\neq j$ and $\pi_F^i\circ \pi_F^i = \pi_F^i$.
\end{proof}

\begin{rmk}
A brute force computation or a universality argument as in \cite{vial} should also show that the above Chow--K\"unneth decomposition is multiplicative.
\end{rmk}

\section{Cubic threefolds and cubic fourfolds}

In this section we prove that the variety of lines on a cubic threefold or a cubic fourfold admits a birational Chow--K\"unneth decomposition when the integer $2$ is inverted.

Let $X\subseteq \PP^{d+1}_{\C}$ be a smooth cubic hypersurface of dimension $d=3$ or $4$. Let $h\in\CH^1(X)$ the class of a hyperplane section. Let $F=F(X)$ be the variety of lines on $X$, which is known to be smooth projective of dimension $2d-4$. For a line $l\subset X$, the corresponding point on $F$ is denoted $[l]$. For any point $t\in F$, the corresponding line is denoted $l_t\subset X$. We define
\[
 I = \Set{(t,t')\in F\times F}{l_t \cap l_{t'} \neq \emptyset}
\]
to be the incidence correspondence. Let
\[
\xymatrix{
 P\ar[r]^q\ar[d]_p &X\\
 F &
}
\]
be the universal family of lines on $X$. This induces homomorphisms
\begin{align*}
 \Phi&=p_*q^*: \CH_1(X)\longrightarrow \CH_2(F)\\
 \Psi &=q_*p^*: \CH_0(F)\longrightarrow \CH_1(X)
\end{align*}
One easily checks that $I_* = \Phi\circ\Psi$. The following result was proved in \cite{relation} for $K$ being algebraically closed. The proof there also works for non-closed fields.

\begin{prop}[\cite{relation, pt}]\label{prop relation}
Let $K\supset \C$ be a field. Then
\[
 \Psi\big( \Phi(\mathfrak{a})\cdot \Phi(\mathfrak{b}) \big) + 2\deg(\mathfrak{a})\mathfrak{b} + 2\deg(\mathfrak{b})\mathfrak{a} - 3\deg(\mathfrak{a})\deg(\mathfrak{b})h^{d-1} = 0,\quad \text{in }\CH_1(X_K)
 \] 
 for all $\mathfrak{a},\mathfrak{b}\in \CH_1(X_K)$, where $\Phi$ and $\Psi$ are their base change to $K$.\hfill $\square$
\end{prop}

\begin{lem}\label{lem o}
The following statements are true.
\begin{enumerate}[(i)]
\item The kernel of $\Psi$ is a uniquely divisible subgroup of $\CH_0(F)$.
\item If $\mathfrak{a}\in \CH_1(X)$ is a torsion element and $\mathfrak{b}\in \CH_1(X)$ is an element of degree 0, then $\Phi(\mathfrak{a})\cdot \Phi(\mathfrak{b}) =0$ in $\CH_0(F)$.
\item If $\mathfrak{o}\in\CH_0(F)$ is an element of degree 1, then 
\[
 I_*\gamma \cdot I_*\mathfrak{o} = -2\gamma,
\]
for every torsion element $\gamma\in \CH_0(F)$.
\item There exists an element $\mathfrak{o}=\mathfrak{o}_F \in \CH_0(F)$ of degree 1 such that
\begin{align*}
 (I_*\mathfrak{o})^2 &= 5\,\mathfrak{o},\quad\text{in }\CH_0(F),\qquad \text{ and }\\
 3 \Psi(\mathfrak{o}) &= h^{d-1}, \quad\text{in }\CH_1(X).
\end{align*}
When $d=3$, the element $\mathfrak{o}$ is unique upto the translation by an element of 3-torsion. When $d=4$, the element $\mathfrak{o}$ is unique.
\end{enumerate}
\end{lem}

\begin{proof}
We first show that $\ker(\Psi)\subset \CH_0(F)$ is torsion free and divisible (and hence uniquely divisible). When $d=3$, the homomorphism $\Psi: \CH_0(F)\rightarrow \CH_1(X)$ factors as the albanese map $\mathrm{alb}:\CH_0(F)\rightarrow \mathrm{Alb}(F)$ followed by an isomorphism $\mathrm{Alb}(F)\overset{\sim}{\longrightarrow} \CH_1(X)$; see \cite{cg}. It follows that $\ker(\Psi)$ is the Albanese kernel. A theorem of Roitman \cite{roitman} says that the Albanese map is an isormphism on the torsion subgroup. Thus we conclude that $\ker(\Psi)$ is torsion free in the cubic threefold case. When $d=4$, the whole group $\CH_0(F)$ is torsion free by Roitman. We still need to show that $\ker(\Psi)$ is divisible. Let $\tau\in \ker(\Psi)$ and let $n$ be a positive integer. Since $\tau$ is of degree zero, there exists some $\tau'\in \CH_0(F)$ such that $n\tau'=\tau$. By assumption we have $n\Psi(\tau')=0$. If $d=4$, then $\CH_1(X)$ is torsion free and hence $\Psi(\tau') =0$. In this case we have $\tau' \in \ker(\Psi)$. If $d=3$, then $\Psi(\tau')$ is an $n$-torsion element of $\CH_1(X)$. By Roitman's theorem, there exists a unique $n$-torsion element $\theta\in\CH_0(F)$ such that $\Psi(\theta) =\Psi(\tau')$. Take $\tau''=\tau'-\theta\in \ker(\Psi)$ and we have $n\tau''=\tau$. This proves (i).

If $\mathfrak{a}\in \CH_1(X)$ is of torsion and $\mathfrak{b}\in \CH_1(X)$ is of degree zero, then we have $\deg(\mathfrak{a}) = \deg(\mathfrak{b}) = 0$. It follows from Proposition \ref{prop relation} that
\[
 \Psi(\Phi(\mathfrak{a})\cdot \Phi(\mathfrak{b})) = 0.
\]
Thus $\Phi(\mathfrak{a})\cdot \Phi(\mathfrak{b})$ is a torsion element of $\ker(\Psi)$. Since $\ker(\Psi)$ is torsion free, we conclude that $\Phi(\mathfrak{a})\cdot \Phi(\mathfrak{b}) =0$, which is statement (ii).

Let $\gamma\in \CH_0(F)$ be a torsion element. Then by Porposition \ref{prop relation}, we have
\[
\Psi(I_*\gamma\cdot I_*\mathfrak{o} + 2\gamma) = \Psi(\Phi(\Psi(\gamma))\cdot\Phi(\Psi(\mathfrak{o}))) + 2\deg(\Psi(\mathfrak{o})) \Psi(\gamma) = 0.
\]
Thus $I_*\gamma\cdot I_*\mathfrak{o} + 2\gamma$ is a torsion element in $\ker(\Psi)$ and hence it is zero.

Now we pick $\gamma\in \CH_0(F)$ such that $3 \Psi(\gamma) = h^{d-1}$. Define
\[
\tau = (I_*\gamma)^2 -5\gamma.
\]
If we write $\mathfrak{a}=\Psi(\gamma)$ and apply Proposition \ref{prop relation}, we get
\[
\Psi(\tau) = \Psi(\Phi(\mathfrak{a})^2) - 5\mathfrak{a} = -4\deg(\mathfrak{a})\mathfrak{a} +3 \deg(\mathfrak{a})^2h^{d-1} - 5\mathfrak{a} = 3(h^{d-1} - 3\mathfrak{a}) =0.
\]
Let $\tau'\in \ker(\Psi)$ be the unique element such that $\tau = 5\tau'$. Let $\mathfrak{o} = \gamma+\tau'$ and note that $I_*\tau' = \Phi(\Psi(\tau'))=0$. We get
\[
 (I_*\mathfrak{o})^2 =(I_*\gamma)^2 = 5\gamma +\tau = 5(\gamma+\tau') =5\mathfrak{o}.
\]
The above argument shows that for every $\mathfrak{a}\in \CH_1(X)$ such that $3\mathfrak{a} = h^{d-1}$, there exists an element $\mathfrak{o}\in \CH_0(F)$ with $\Psi(\mathfrak{o}) = \mathfrak{a}$ and satisfies the required equations. Actually, such an $\mathfrak{o}$ lifting $\mathfrak{a}$ is unique. Indeed, if $\mathfrak{o}'$ is a different choice, then we have $\mathfrak{o}' = \mathfrak{o} + \tau$ for some $\tau\in \ker(\Psi)$. Then $I_*\mathfrak{o}' = I_*\mathfrak{o} + \Phi(\Psi(\tau)) = I_*\mathfrak{o}$ and hence
\[
0=5\mathfrak{o}' - (I_*\mathfrak{o}')^2 = 5\mathfrak{o} + 5\tau -(I_*\mathfrak{o})^2 = 5\tau.
\]
We conclude that $\tau=0$ since $\ker(\Psi)$ is torsion free. Then the uniqueness of $\mathfrak{o}$ in the case $d=4$ follows since $\CH_1(X)$ is torsion free and $\mathfrak{a}\in \CH_1(X)$ is unique. We assume that $d=3$. Let $\mathfrak{o}\in \CH_0(F)$ be an element satisfying the required equtions. Take $\mathfrak{o}'=\mathfrak{o}+\alpha$, where $\alpha\in\CH_0(F)$ is of 3-torsion. Since $\Psi(\alpha)$ is of 3-torsion and hence by (ii) we have $(I_*\alpha)^2=0$. Thus we have
\[
\big( I_*(\mathfrak{o}+\alpha) \big)^2 = (I_*\mathfrak{o})^2 + 2 I_*\alpha\cdot I_*\mathfrak{o}= 5\mathfrak{o} +2 (-2)\alpha = 5(\mathfrak{o} + \alpha) - 9\alpha = 5(\mathfrak{o}+\alpha).
\]
This shows that $\mathfrak{o}'$ also satisfies the equations.
\end{proof}

\begin{rmk}
When $d=4$, the unique element $\mathfrak{o}\in \CH_0(F)$ obtained above is the same as the canonical class obtained by Voisin \cite{voisin hk}. This class was also essential to the multiplicative decomposition obtained in \cite{shen-vial}.
\end{rmk}

\begin{thm}\label{thm generic cubic}
Let $X\subseteq\PP^{d+1}_{\C}$ be a smooth cubic hypersurface of dimension $d=3$ or $4$ and let $F$ be its variety of lines. Let $\mathfrak{o}\in F$ be an element of degree one as in (iv) of Lemma \ref{lem o}. We take
\[
\varpi_F^0 : = \mathfrak{o}\times F,\quad \varpi_F^{d-2} = \frac{1}{2}\big(5\,\mathfrak{o}\times F - p_1^*(I_*\mathfrak{o})\cdot I\big), \quad \varpi_F^{2d-4} = \delta_F - \varpi_F^0 - \varpi_F^{d-2}
\]
where all cycles are understood to be their restriction to $F\times \eta_F$ and we set $\varpi_F^j=0$ for all $j$ satisfying $1\leq j \leq 2d-5$ and $j\neq d-2$. Then $\varpi_F^i$, $0\leq i\leq 2d-4$, form a birational Chow--K\"unneth decomposition with $\Z[\frac{1}{2}]$-coefficients.
\end{thm}

\begin{proof}
It is clear that $\varpi_F^0\circ\varpi_F^0 = \varpi_F^0$. To compute the other compositions, we make the following observation. Let $\gamma\in \CH_0(F\times \eta_F)$ and let $\gamma_K\in \CH_0(F_K\times_K \eta_{F_K})$ be the base change of $\gamma$ to a field $K\supset \C$. When we take $K=\C(F)$, then one easily checks
\[
\big( \gamma\otimes \C(F)\big)^*\delta_F = \gamma.
\]
To make the notation simpler, we simply write $\gamma^*\delta_F = \gamma$. Now for any $\tau\in \CH_0(F_K)$, we have
\begin{align*}
 \Big(\varpi_F^0\circ p_1^*(I_*\mathfrak{o})\cdot I \Big)^*\tau
  & = (\mathfrak{o}\times F)^* \big( p_1^*(I_*\mathfrak{o})\cdot I \big)^* \tau\\
  & = (\mathfrak{o}\times F)^*\big( I_*\mathfrak{o}\cdot I_*\tau\big)\\
  & = \deg\big( I_*\mathfrak{o}\cdot I_*\tau\big) \mathfrak{o}\\
  & = 5\deg(\tau) \mathfrak{o}\\
  & = 5(\varpi_F^0)^*\tau.
\end{align*}
Similarly, we have
\[
\big(p_1^*(I_*\mathfrak{o})\cdot I \circ \varpi_F^0\big)^*\tau =  \deg(\tau) (I_*\mathfrak{o})^2 = 5\deg(\tau)\mathfrak{o} = 5(\varpi_F^0)^*\tau.
\]
In particular, by taking $K=\C(F)$ and $\tau = \delta_F$, we have
\[
 \varpi_F^0\circ (p_1^*(I_*\mathfrak{o})\cdot I) =  (p_1^*(I_*\mathfrak{o})\cdot I)\circ \varpi_F^0 = 5 \varpi_F^0.
\]
From this, we easily get
\[
\varpi_F^0\circ \varpi_F^{d-2} = \varpi_F^{d-2}\circ\varpi_F^0 = 0.
\]
For any element $\tau\in\CH_0(F_K)$, we have
\begin{align*}
\Big( (p_1^*(I_*\mathfrak{o})\cdot I)\circ (p_1^*(I_*\mathfrak{o})\cdot I) \Big)^* \tau & = I_*\mathfrak{o} \cdot I_*\Big( I_*\mathfrak{o}\cdot I_*\tau\Big)\\
 & = I_*\mathfrak{o} \cdot \Phi \Psi(\Phi(\mathfrak{a})\cdot \Phi(\mathfrak{b})),\qquad \mathfrak{a} = \Psi(\mathfrak{o}),\,\, \mathfrak{b} = \Psi(\tau)\\
 & = I_*\mathfrak{o} \cdot \Phi\Big(-2\deg(\mathfrak{b})\mathfrak{a} - 2\deg(\mathfrak{a})\mathfrak{b} +3\deg(\mathfrak{a}\deg(\mathfrak{b})h^{d-1}) \Big)\\
 & = I_*\mathfrak{o}\cdot \Big( -2\deg(\tau)I_*\mathfrak{o} -2 I_*\tau + 9\deg(\tau) I_*\mathfrak{o}\Big)\\
 & = 35 \deg(\tau)\mathfrak{o} -2 I_*\mathfrak{o}\cdot I_*\tau\\
 & = 35 (\varpi_F^0)^*\tau - 2 (p_1^*(I_*\mathfrak{o})\cdot I)^*\tau.
\end{align*}
From this we get
\[
(p_1^*(I_*\mathfrak{o})\cdot I)\circ (p_1^*(I_*\mathfrak{o})\cdot I) = 35 \varpi_F^0 - 2 p_1^*(I_*\mathfrak{o})\cdot I.
\]
From this equation we get
\[
\varpi_F^{d-2}\circ\varpi_F^{d-2} = \frac{1}{4}(25\varpi_F^0 -2\cdot 5\cdot 5 \varpi_F^{0} + 35\varpi_F^0 -2 p_1^*(I_*\mathfrak{o})\cdot I) = \varpi_F^{d-2}.
\]
Thus the $\varpi_F^i$ are all projectors. 

We need to show that the action of $\varpi_F^i$ on cohomology is the expected one. Now it is clear that $[\varpi_F^0]$ is the projector onto $\HH^0(F)$. We only need to show that the class $[\varpi_F^{d-2}]$ is the expected one. Note that the cohomology class of $\varpi_F^{d-2}$ is given by
\[
[\varpi_F^{d-2}] = \frac{1}{2}\Big( 5 [\mathfrak{o} \times F] - [p_1^*S_l\cdot I]\Big)
\]
where $l\subset X$ is a general lines and $S_l\subset F$ is the subvariety of all lines meeting the given line $l$. As a consequence, the cohomological action of $\varpi_F^{d-2}$ factors through the cohomology of $S_l$. Thus we see that $\varpi_F^{d-2}$ acts trivially on $\HH^{2d-4,0}(F)$. It is also clear that the action of $\varpi_F^{d-2}$ on $\HH^{0,0}(F)$ is zero. Let $\alpha\in \HH^{d-2,0}(F)$. Since $\Phi=p_*q^*: \HH^d(X)\longrightarrow  \HH^{d-2}(F)$ is an isomorphism; see \cite{cg, bd}. We see that there exists a unique $\omega\in\HH^{d-1,1}(X)$ such that $\alpha = \Phi(\omega)$. Thus we have
\[
[\varpi_F^{d-2}]_*\alpha =-\frac{1}{2} \Phi\big( \Psi(S_l\cdot\Phi(\omega)) \big) = -\frac{1}{2}\Phi(-2\omega) = \alpha.
\]
Here we used the fact that $\Psi(S_l\cdot\Phi(\omega)) = -2\omega$, which is a direct consequence of the formula
\[
 S_l\cdot \Phi(\tau)\cdot\Phi(\tau') = -2\langle \tau,\tau'\rangle_X,\quad\text{for all }\tau,\tau'\in\HH^{d}(X)_{\mathrm{tr}}.
\]
We conclude the proof since $\HH^{j,0}(F)=0$ for all $j\neq 0,d-2,2d-4$.
\end{proof}

\begin{rmk}
When $d=4$, the induced decomposition of $\CH_0(F)$ is the one obtained in \cite{shen-vial}. When $d=3$, we get a decomposition
\[
\CH_0(F) = \Z\mathfrak{o}_F \oplus \CH_0(F)_1 \oplus \CH_0(F)_2,
\]
where $\CH_0(F)_1$ and $\CH_0(F)_2$ are independant of the choice of $\mathfrak{o}_F$.
\end{rmk}

\section{Stably rational cubic threefolds and cubic fourfolds}

In this section we show that a birational Chow--K\"unneth decomposition exists on the variety of lines on a stably rational cubic threefold or a stably rational cubic fourfold.

\begin{thm}\label{thm rational cubic}
Let $X\subset \PP^{d+1}_{\C}$ be a smooth cubic hypersurface of dimension $d=3$ or $4$. Let $F$ be the variety of lines on $X$. Assume that $X$ has a cohomological decomposition of the diagonal. Then $F$ admits a birational Chow--K\"unneth decomposition.
\end{thm}

\begin{proof}
In \cite[Theorem 1.3]{inthdg}, we showed that there exist finitely many smooth projective varieties $Z_i\rightarrow F$ of dimension $d-2$, correspondences $\Gamma_i = P|_{Z_i} \in \CH_{d-1}(Z_i\times X)$ and symmetric self-corresponcences $\sigma_i\in \CH_{d-2}(Z_i\times Z_i)$ such that
\begin{equation}
\Delta_X = \sum \Gamma_i\circ\sigma_i\circ {}^t\Gamma_i + a (h\otimes l + l \otimes h) + \Xi + (x\times X + X\times x),\quad\text{in }\HH^{2d}(X\times X,\Z),
\end{equation}
for some $a\in\Z$. The cycle $\Xi$ is symmetric and non-zero only when $d=4$, in which case $\Xi = \sum S_i\otimes S_j$ with $S_i,S_j\in \CH_2(X)$ being $2$-cycles.

By \cite[Proposition 5.1]{inthdg}, we may enlarge the collection $(\Z_i,\Gamma_i,\sigma_i)$ and assume that
\[
\Delta_X \sim_{\mathrm{alg}} \sum \Gamma_i\circ\sigma_i\circ {}^t\Gamma_i + a (h\otimes l + l \otimes h) +\Xi+ (x\times X + X\times x).
\]
By the main results of \cite{v nil, voisin nil}, we see that
\begin{equation}\label{eq N power}
 \left( \Delta_X - \sum \Gamma_i\circ {}^t\Gamma_i - a (h\otimes l + l \otimes h) -\Xi- (x\times X + X\times x)\right)^{\circ N} =0,\quad \text{in }\CH_d(X\times X),
\end{equation}
for some sufficiently large $N$. To simplify the expansion of the above expression, we note that
\begin{align*}
(\Gamma_i\circ\sigma_i\circ {}^t\Gamma_i)\circ (D\otimes C) & = D\otimes C',\\
(\Gamma_i\circ\sigma_i\circ {}^t\Gamma_i)\circ (C\otimes D) & = C\otimes D',\\
(D\otimes C)\circ (\Gamma_i\circ\sigma_i\circ {}^t\Gamma_i) & = D'\otimes C,\\
(C\otimes D)\circ (\Gamma_i\circ\sigma_i\circ {}^t\Gamma_i) & = C'\otimes D,
\end{align*}
where $D,D'\in \CH^1(X)$ and $C,C'\in \CH_1(X)$. When $d=4$, we also have
\[
 (\Gamma_i\circ\sigma_i\circ {}^t\Gamma_i)\circ (S_1\otimes S_2)  = S_1\otimes S'_2\quad \text{and}\quad (S_1\otimes S_2)\circ(\Gamma_i\circ\sigma_i\circ {}^t\Gamma_i) = S'_1\otimes S_2,
\] 
 where $S_1,S_2,S'_1,S'_2\in \CH_2(X)$. Then the terms in the expansion of the equation \eqref{eq N power} consists of the following types.

\textit{Type 1}: $\Gamma\circ\sigma \circ{}^t\Gamma$. Here $\sigma \in\CH_{d-2}(T\times T)$ is a symmetric self-correspondence on a smooth projective variety $T\rightarrow F$ of dimension $d-2$ and $\Gamma=P|_{T}\in \CH_{d-1}(T\times X)$ is the restriction of the universal line.

\textit{Type 2}: $\Gamma\circ\sigma\circ {}^t\Gamma' + \Gamma'\circ{}^t\sigma\circ{}^t\Gamma$. Here $\sigma \in\CH_{d-2}(T'\times T)$ is a correspondence between $(d-2)$-dimensional smooth projective varieties $T\rightarrow F$ and $T'\rightarrow F$ and $\Gamma = P|_{T\times X}\in \CH_{d-1}(X)$ and $\Gamma'=P_{T'\times X}\in \CH_{d-1}(T'\times X)$ are the restrictions of the universal line.

\textit{Type 3}: $D\otimes C + C\otimes D$. Here $D\in \CH^1(X)$ is a divisor class and $C\in \CH_1(X)$ is a 1-cycle.

\textit{Type 4}: $x\times X + X\times x$. This term is unique and appears with coefficient 1.

\textit{Type 5}: (only for $d=4$) $S\otimes S' + S'\otimes S$, where $S,S'\in \CH_2(X)$.

As a consequence we have
\begin{equation}
\Delta_X = \sum_{\text{type 1}} \Gamma\circ\sigma\circ{}^t\Gamma + \sum_{\text{type 2}} (\Gamma\circ\sigma\circ{}^t\Gamma' + \Gamma'\circ {}^t\sigma\circ{}^t\Gamma) + \sum_{\text{type 3}} (D\otimes C + C\otimes D) + (x\times X + X\times x) +\Xi
\end{equation}
in $\CH_d(X\times X)$, where $\Xi = \sum S_i\otimes S_j$.

For each $T$, we use $f:T\rightarrow F$ to denote the morphism of $T$ to $F$. We define
\[
\theta : = \sum_{\text{type 1}}f_*\circ\sigma \circ f^* + \sum_{\text{type 2}} (f_*\circ\sigma\circ f'^* + f'_*\circ{}^t\sigma\circ f^*), \quad \text{in }\CH_{d-2}(F\times F).
\]
It fllows that
\begin{align*}
\Psi\circ\theta \circ \Phi & = \sum_{\text{type 1}} \Gamma\circ\sigma\circ{}^t\Gamma + \sum_{\text{type 2}} (\Gamma\circ\sigma\circ{}^t\Gamma' + \Gamma'\circ{}^t\sigma\circ{}^t\Gamma) \\
 & = \Delta_X -\sum_{\text{type 3}} (D\otimes C + C\otimes D) - (x\times X + X\times x) -\Xi.
\end{align*}
Let
\[
\Pi^{d-2}_F = \Phi\circ\Psi\circ\theta
\]
and we have
\begin{align*}
\Pi^{d-2}_F\circ \Pi^{d-2}_F & = \Phi\circ\Psi\circ\theta\circ\Phi\circ\Psi\circ\theta\\
 & = \Phi\circ \left(\Delta_X -\sum_{\text{type 3}} (D\otimes C + C\otimes D) - (x\times X + X\times x) -\Xi \right)\circ \Psi \circ \theta\\
 & = \Pi_F^{d-2} -\Phi\circ\left( \sum_{\text{type 3}} (D\otimes C + C\otimes D) + (x\times X + X\times x) + \Xi \right)\circ\Psi\circ\theta
\end{align*}
Now we restrict the above equation to $F\times \eta_F$ and we get
\begin{equation}\label{eq varpi1}
\varpi'^{d-2}_F \circ \varpi'^{d-2}_F = \varpi'^{d-2}_F + \alpha\otimes \eta_F,\quad\text{in }\CH_0(F_{\C(F)}),
\end{equation}
where $\varpi'^{d-2}_F$ is the restriction of $\Pi_F^{d-2}$ to $F\times \eta_F$ and $\alpha\in \CH_0(F)$. Note that
\[
\varpi'^{d-2}_F\circ (\mathfrak{o}_F\times\eta_F) = r \mathfrak{o}_F\times \eta_F \quad \text{and}\quad (\mathfrak{o}_F \times \eta_F) \circ \varpi'^{d-2}_F = (\varpi'^{d-2}_F)^*\mathfrak{o}_F \times \eta_F
\]
where $r=\deg((\varpi'^{d-2}_F)^*\mathfrak{o}_F)$. Take $\gamma=(\varpi'^{d-2}_F)^*\mathfrak{o}_F\in \CH_0(F)$. We modify $\varpi'^{d-2}_F$ a little by taking
\[
\varpi^{d-2}_F = \varpi'^{d-2}_F - \gamma\otimes \eta_F.
\]
Then $\varpi^{d-2}_F$ satisfies the following equations
\begin{align}
\label{eq alpha'} \varpi^{d-2}_F\circ\varpi_F^{d-2} &= \varpi_F^{d-2} + \alpha'\times \eta_F,\quad\alpha'\in \CH_0(F)\\
 \label{eq varpi2}\varpi^{d-2}_F\circ \mathfrak{o}_F\times\eta_F & = \mathfrak{o}_F\times \eta_F\circ\varpi^{d-2}_F =0.
\end{align}
We use equation \eqref{eq alpha'} to pull back the class $\mathfrak{o}_F$ and use the fact that $(\varpi^{d-2}_F)^*\mathfrak{o}_F=0$, we get $\alpha'=0$. Thus $\varpi^{d-2}_F$ is projector whose action on $\HH^0(F)$ is zero by construction. We check that its action on $\HH^{d-2,0}(F)$ is identity. This follows from the following computation. Let $u\in \HH^{d-2,0}(F)$, then we have $u =\Phi(v)$ for some $v\in \HH^{d-1,1}(X)$. Thus
\begin{align*}
(\varpi_F^{d-2})_*u &= \Phi\circ\Psi\circ\theta\circ\Phi (v) \\
 & = \Phi \circ\left( \Delta_X -\sum_{\text{type 3}} (D\otimes C + C\otimes D) - (x\times X + X\times x) -\Xi\right)_*v\\
 & = \Phi((\Delta_X)_*v)\\
 & = u.
\end{align*}
Since $\varpi_F^{d-2}$ factors through varieties of dimension $d-2$ (namely the $T$'s), we see that its action on $\HH^{2d-4,0}(F)$ is trivial. Hence we see that
\[
\varpi_F^0 = \mathfrak{o}\times \eta_F, \quad \varpi_F^{d-2}\quad \text{and}\quad \varpi^{2d-4}_F=\delta_F- \varpi_F^0-\varpi_F^{d-2}
\]
give a birational Chow--K\"unneth decomposition on $F$.
\end{proof}

\begin{rmk}
We combine the above argument together with the results of \cite{inthdg} and get the following. If $X$ admits a cohomological decomposition of the diagonal, then there exists a symmetric cycle $\theta\in \CH_{d-2}(F\times F)$, such that
\begin{enumerate}
\item The following numerical propterty holds
\[
 \theta \cdot \alpha\cdot \beta = -\mathfrak{B}(\alpha,\beta),\quad \text{for all }\alpha,\beta\in \HH^{d-2}(F)_{\mathrm{tr}},
\]
where $\mathfrak{B}$ is the natural bilinear form on $\HH^{d-2}(F)$ (the principal polarization when $d=3$ and the Beauville--Bogomolov form when $d=4$).
\item The correspondence $I\circ \theta$ is essentially a birational projector $\varpi_F^{d-2}$, where $I=\Phi\circ \Psi$ is the incidence correspondence.
\end{enumerate}
\end{rmk}

\end{document}